\documentclass{birkjour}
\usepackage{graphicx}
\usepackage{comment}

 \newtheorem{thm}{Theorem}[section]
 \newtheorem{cor}[thm]{Corollary}
 \newtheorem{proposition}[thm]{Proposition}
 \theoremstyle{definition}
 \newtheorem{defn}[thm]{Definition}
 \numberwithin{equation}{section}

\begin{document}

\title[Firm non-expansive mappings in weak metric spaces]
 {Firm non-expansive mappings\\ in weak metric spaces}

\author{Armando W. Guti\'{e}rrez}

\address{%
Inria and CMAP, Ecole Polytechnique\\
CNRS, 91128, Palaiseau\\
France}
\email{armando.w.gutierrez@inria.fr}

\author{Cormac Walsh}
\address{%
Inria and CMAP, Ecole Polytechnique\\
CNRS, 91128, Palaiseau\\
France}
\email{cormac.walsh@inria.fr}




\begin{abstract}
We introduce the notion of firm non-expansive mapping in weak metric spaces,
extending previous work for Banach spaces and certain geodesic spaces.
We prove that, for firm non-expansive mappings, the minimal displacement, 
the linear rate of escape, and the asymptotic step size are all equal.
This generalises a theorem by Reich and Shafrir.
\end{abstract}

\maketitle

\section{Introduction}

A fundamental question in the theory of metric spaces is the long term behaviour of iterates 
of non-expansive mappings. Recall that a mapping $T$ of a metric space $(X,\delta)$ into
itself is said to be \emph{non-expansive} if, for every $x,y\in X$,
\[ 
	\delta(Tx,Ty) \leq \delta(x,y). 
\]
Particularly interesting is the case where the mapping has no fixed point, because here the 
iterates have the possibility of escaping to infinity.

In the setting of Banach spaces, Bruck~\cite{B1973} introduced a special class of non-expansive 
mappings, which he called \emph{firmly} non-expansive. These mappings were further studied 
in~\cite[p.41,~p.129]{GR1984} and~\cite{RS1987}. In particular, it was shown in~\cite{RS1987} 
that for firmly non-expansive mappings the minimal displacement, the linear rate of escape, and 
the asymptotic step size are all equal. The definition of firmly non-expansive and the equality of 
these three quantities were extended by Ariza-Ruiz et al.~\cite[Theorem~5.1]{ArLLo2014} to a 
broader class of spaces, namely, the $W$-hyperbolic spaces. These are geodesic metric spaces 
with a certain ``negative curvature''-type condition.

In this paper, we broaden the idea further. We introduce the notion of \emph{firm non-expansive} 
mapping in arbitrary (weak) metric spaces. Our definition does not assume the existence of 
geodesics. This is significant since in modern optimisation applications one often deals with 
discrete spaces or has access to a collection of points of a space whose geometric structure is 
unknown. We show, in Section~\ref{sec:2}, that our class of firm non-expansive mappings contains 
all the firmly non-expansive ones in the setting of Banach spaces or $W$-hyperbolic spaces. 
In Section~\ref{sec:3}, we provide non-trivial examples by characterising the firm non-expansive 
mappings of $1$-dimensional asymmetric normed spaces. 
We then prove in Section~\ref{sec:4} our main result, which is as follows.

\begin{thm}\label{th:1}
Let $T$ be a firm non-expansive mapping of a weak metric space $(X,\delta)$ into itself. 
Then, for every $x\in X$ and every integer $k\geq 1$,
\begin{equation*}
\resizebox{\textwidth}{!}{$\displaystyle \inf_{w \in X} \delta(w,Tw) = \lim_{n\to\infty}\frac{\delta(x,T^{n}x)}{n} 
		= \lim_{n\to\infty}\delta(T^{n}x,T^{n+1}x) 
		= \frac{1}{k}\lim_{n\to\infty}\delta(T^{n}x,T^{n+k}x).$}
\end{equation*}
\end{thm}	

The first three quantities in this equation are, respectively, the minimal displacement, the linear rate 
of escape, and the asymptotic step size, mentioned earlier. Our theorem thus generalises the results 
by Reich and Shafrir and by Ariza-Ruiz et al., referred to above.

We conclude this note by proving in Section~\ref{sec:5} a couple of corollaries of our main result, 
concerning the behaviour of the iterates of firm non-expansive mappings in terms of \emph{metric functionals}.

The first enhances \cite[Theorem~16]{GV2012} for firm non-expansive mappings in arbitrary metric 
spaces when the linear rate of escape, denoted by $\rho(T)$, equals zero (compare also 
with \cite[Proposition~3]{GK2021}).

\begin{cor}\label{cor:2}
Let $T$ be a firm non-expansive mapping of a metric space $(X,\delta)$ into itself 
such that $\rho(T)=0$. Then, there exists a metric functional $h$ on $X$ such that, 
for every $x\in X$,
\begin{equation*} 
	h(Tx) \leq h(x).
\end{equation*}
\end{cor}

We point out that the metric functional $h$ in Corollary~\ref{cor:2} is defined in terms of the 
iterates $T^{n}x_{0}$ for some $x_{0}\in X$, whereas the construction of the metric functional 
appearing in \cite[Theorem~16]{GV2012} relies strongly on the non-positive curvature condition 
assumed there.

A theorem by Lins~\cite[Theorem~2.1]{L2009} states that, for every fixed-point-free non-expansive 
mapping in a finite-dimensional normed space, there is a metric functional along which the orbits are 
seen to escape to infinity. It is not known in general whether this behaviour of escaping to infinity is 
monotone. We use Corollary~\ref{cor:2} to show that, in the case of firm non-expansive mappings, it is.

\begin{cor}\label{cor:3}
Let $T$ be a firm non-expansive mapping of a finite-dimensional normed space $(V,\| \cdot \|)$ 
into itself. Suppose that $\rho(T)=0$. Then, either $T$ has bounded orbits, or there is a metric 
functional $h$ on $V$ such that, for every $x\in V$, the sequence $(h(T^{n}x))_n$ converges 
monotonically to $-\infty$. Moreover, the metric functional $h$ is a limit point of the orbit 
$(T^{n}0)_{n \geq 0}$.
\end{cor}

\section{Firm non-expansive mappings}\label{sec:2}

A \emph{weak metric} on a set $X$ is a mapping $\delta \colon X \times X \to [0, +\infty[$ 
satisfying $\delta(x, x) = 0$, for every $x\in X$, and the triangle inequality 
\[\delta(x,z) \leq \delta(x,y) + \delta(y,z),\] 
for every $x,y,z \in X$. Examples of weak metrics include the \emph{Funk weak metric}, 
the \emph{Apollonian weak metric}, and \emph{Thurston's metric}; see the papers \cite{PT2007,PT2009} 
by Papadopoulos and Troyanov for more details.
Weak metrics are sometimes called \emph{quasi-pseudometrics}.

Motivated by \'{C}iri\'{c}'s work \cite{C1971} on a generalisation of Banach's contraction principle, 
we introduce the notion of firm non-expansive mapping in an arbitrary weak metric space.

\begin{defn}
A mapping $T$ of a weak metric space $(X,\delta)$ into itself is \emph{firm} if there exist 
mappings $q$, $r$, $s$, and $t$ of $X \times X$ into the interval $[0,+\infty[$ satisfying
the following three conditions:

\begin{equation}\label{eq:A}
	\inf_{(x,y) \in X \times X} t(x,y) \;>\;0;
\end{equation}
\begin{equation}\label{eq:B}
	\sup_{(x,y) \in X \times X} q(x,y) + r(x,y) + s(x,y) + 2t(x,y) \;\leq\;1;
\end{equation}
and, for every $x,y\in X$, 
\begin{equation}\label{eq:C}
	\begin{split}
		\delta(Tx,Ty) &\leq  q(x,y)\delta(x,y) 
			+ r(x,y)\delta(x,Tx)
			+ s(x,y)\delta(y,Ty) \\
			&\quad +\,t(x,y)\Big[\delta(x,Ty)+\delta(Tx,y)\Big].
	\end{split}
\end{equation} 	

We say that $T$ is \emph{firm non-expansive} if it is both firm and non-expansive.
\end{defn}

For purposes of comparison, we recall Bruck's notion of firmly non-expansive mapping. A mapping $T$ 
of a Banach space $(V,\|\cdot\|)$ into itself is said to be firmly non-expansive if, for every $x,y\in V$ and 
every $0< \lambda \leq 1$, 
\begin{equation}\label{eq:fnenorm}
	\|Tx-Ty\| \leq \|(1-\lambda)(Tx-Ty)+\lambda(x-y)\|.
\end{equation}	
A simple example of a firm non-expansive mapping that is not firmly non-expansive is the mapping 
$x\mapsto Tx=\left| x \right| +1$ on the real line endowed with its usual metric. Indeed, for $x<0$ and $y=0$, 
there is no $\lambda\in]0,1[$ such that (\ref{eq:fnenorm}) holds. One can readily verify that this mapping 
is firm non-expansive. Another example of a firm non-expansive mapping on the real line is
\[
	Tx = \begin{cases}
			-x+1 & \mbox{if } x < 0, \\
			x+e^{-x} & \mbox{if } x \geq 0.
		\end{cases}	
\] 
By setting $q(x,y)=r(x,y)=s(x,y)=0$ and $t(x,y)=1/2$ for every $x,y$, we see that $T$ is firm non-expansive.  

Every firmly non-expansive mapping in a Banach space is firm non-expansive. Indeed, if $T$ satisfies 
inequality~(\ref{eq:fnenorm}) with $\lambda\in]0,1[$, then, for every $x,y\in V$,
\[
	\| Tx-Ty \| \leq \frac{\lambda}{2-\lambda}\| x-y \| 
			+ \frac{1-\lambda}{2-\lambda}\left( \| x-Ty \| + \| Tx-y \| \right).
\]
Our claim follows from this by fixing $\lambda \in ]0, 1[$ and setting
\begin{align*}
	q(x,y)=\frac {\lambda} {2-\lambda}, \qquad
	r(x,y)=0, \qquad
	s(x,y)=0, \qquad \text{and} \quad
	t(x,y)=\frac{1-\lambda}{2-\lambda},
\end{align*}
for every $x,y\in V$. More generally, every firmly non-expansive mapping in a $W$-hyperbolic space is 
firm non-expansive. This follows immediately from \cite[Lemma~5.6]{ArLLo2014}.

Our definition of firm non-expansive is inspired by \'{C}iri\'{c}'s definition \cite{C1971} of a 
\emph{generalised contraction}. This is a mapping $T$ from a metric space $(X, \delta)$ to itself such 
that there exist mappings $q$, $r$, $s$, and $t$ from $X \times X$ to $[0, +\infty[$ such that (\ref{eq:C}) 
above holds, and the supremum in (\ref{eq:B}) is strictly less than $1$. Note that when this is the case, 
it is possible to slightly increase $t$ everywhere so that (\ref{eq:A}) also holds. The relation between 
generalised contractions and our firm non-expansive mappings may be considered similar to the relation 
between strict contractions and non-expansive mappings. \'{C}iri\'{c} showed that generalised contractions 
satisfy the conclusion of Banach's Contraction Theorem.
 
Other authors have focused on the existence of a unique fixed point
of mappings that satisfy conditions similar to those considered
in our definition of firm mapping; see~\cite{GKS1973,HR1973,R1971,R1977,W1974}.
In contrast, we are interested in the asymptotic behaviour
of fixed-point-free non-expansive mappings.

The definition of firmly non-expansive mappings of Ariza-Ruiz et al.\ was studied in the context of CAT(0) 
spaces in~\cite{LNS2018}. For a different approach to generalising firmly non-expansive mappings, 
see~\cite{BLL2021}.

\newcommand\R{\mathbb{R}}

\section{An example: asymmetric norms on $\R$.}\label{sec:3}

The following characterisation of firmness for
non-expansive mappings will prove useful in studying our example.
Let $T$ be a non-expansive self-mapping of a weak metric space $(X, \delta)$.
For each $x, y \in X$, define
\begin{align*}
M(x, y) &:= \max\Big[\delta(x, y), \delta(x, Tx), \delta(y, Ty)\Big] \\
A(x, y) &:= \frac {1} {2} \Big(\delta(x, Ty) + \delta(Tx, y)\Big) \\
\text{and} \qquad
\tau(x, y) &:= \frac {M(x, y) - \delta(Tx, Ty)} {2 \big(M(x, y) - A(x, y)\big)}.
\end{align*}
Then, $T$ is firm if and only if
\begin{align*}
\inf \Big\{ \tau(x, y)
    \mid \text{$x, y\in X$,  $A(x, y) < \delta(Tx , Ty)$}\Big\} > 0.
\end{align*}

\subsection{Asymmetric norms on $\R$.}

We consider $\R$ endowed with an asymmetric norm
\begin{align*}
||x|| := \max\big(-\alpha x, \beta x\big),
\end{align*}
where $\alpha$ and $\beta$ are positive real numbers.

For each distinct $x$ and $y$ in $\R$, define $z_{xy}$ to be the unique
element of $\R$ such that $||z_{xy} - y|| = ||y - x||$
and $(y - x)(z_{xy} - y) < 0$.
If $x$ and $y$ are points of $\R$, and $T$ is a mapping such that $Tx = y$ and
$Ty = z_{xy}$, then one can calculate that $\tau(x,y) = 0$
in the criterion above with $\delta(x,y):=||y-x||$, and hence the mapping $T$ is not firm.

The following proposition shows that a non-expansive self-mapping of $(\R, ||\cdot||)$
fails to be firm precisely when one can find points arbitrarily close to this
configuration.

\begin{proposition}
Let $(\R, ||\cdot||)$ be the asymmetric normed space above.
Then, a non-expansive mapping $T \colon \R \to \R$ is not firm
if and only if, for every $\epsilon >0$ there exist distinct points
$x$ and $y$ in $\R$ such that $||Tx - y||$ and $||Ty - z_{xy}||$
are both less than $||y - x||\epsilon$.
\end{proposition}

In particular, we see that every fixed-point-free non-expansive mapping
on $(\R, ||\cdot||)$ is firm. This shows that firm non-expansive mappings
are considerably more general than firmly non-expansive mappings,
even in dimension one. Recall that, on $\R$, the latter mappings are precisely
the $1$-Lipschitz mappings that are non-decreasing.

\begin{proof}
Assume first that $T$ satisfies the condition.
So, we can find sequences of points $x_n$ and $y_n$ in $\R$ such that
$||Tx_n - y_n|| / d_n$ and $||Ty_n - z_{x_n, y_n}|| / d_n$
converge to zero, where $d_n := ||y_n - x_n||$, for all $n$.
It follows that $M_n / d_n$ and $\delta_n / d_n$ both
converge to $1$, where $M_n := M(x_n, y_n)$ and $\delta_n := ||Ty_n - Tx_n||$,
for all $n$. Define $A_n := A(x_n, y_n)$, for all $n$.
Since $||z_{x_n, y_n} - x_n|| \le ||y_n - x_n||$, for all $n$,
we have
\begin{align*}
\limsup_n \frac {A_n} {d_n} \le \frac {1} {2}.
\end{align*}
This gives $A_n < \delta_n$ for $n$ large enough,
and that $\tau(x_n, y_n)$ converges to zero.
By the criterion at the start of this section, $T$ is not firm.

Now assume that $T$ is not firm.
So, there exists sequences of points $x_n$ and $y_n$ in $\R$ such that
$A_n < \delta_n$ for all $n$, and $\tau_n$ converges to zero,
where we have written $A_n := A(x_n, y_n)$ and $\delta_n := ||Ty_n - Tx_n||$,
and $\tau_n := \tau(x_n, y_n)$.
We assume without loss of generality that $x_n < y_n$, for all $n$. If there
is a subsequence where the opposite inequality is true, it can be handled in
a similar manner.
Note that, for any $n$,
\begin{align*}
A_n \ge  \frac {1} {2} \beta \big(Ty_n - x_n + y_n - Tx_n \big).
\end{align*}
If $Ty_n \ge Tx_n$, then the non-expansiveness of $T$ would imply that
the quantity on the right-hand-side is greater than or equal to $\delta_n$,
which is not the case.
We conclude that $Ty_n < Tx_n$, for all $n$.

For each $n$,
\begin{align*}
\tau_n \ge \frac {M_n - \delta_n} {2M_n}
       =   \frac {1} {2} \Big(1 - \frac {\delta_n} {M_n} \Big)
\qquad\text{and}\quad
\delta_n \le ||y_n - x_n|| \le M_n,
\end{align*}
where we have defined $M_n := M(x_n, y_n)$.
We deduce that the ratios
\begin{align*}
 \delta_n / ||y_n - x_n|| 
 \qquad\text{and}\quad 
 M_n / ||y_n - x_n||
\end{align*}
both converge to $1$, as $n$ tends to infinity.

It follows from the convergence of the latter ratio that $(u_n, v_n)$
stays within a bounded region of the plane, where
\begin{align*}
u_n := \frac {Tx_n - x_n} {||y_n - x_n||}
\qquad\text{and}\quad
v_n := \frac {Ty_n - y_n} {||y_n - x_n||}.
\end{align*}
Let $(u, v)$ be a limit point of this sequence.
Again from the convergence of the ratio, both $u$ and $v$
lie in the interval $[-1/\alpha, 1/\beta]$.
Since, for each $n$,
\begin{align*}
\frac {\delta_n} {||y_n - x_n||}
    = \frac {||Ty_n - Tx_n||} {||y_n - x_n||}
    &= -\alpha \Big(v_n - u_n + \frac {y_n - x_n} {||y_n - x_n||} \Big) \\
    &= -\alpha \Big(v_n - u_n + \frac {1} {\beta} \Big),
\end{align*}
we get that $v - u = -1/\alpha - 1/\beta$.
We deduce that $v = -1/\alpha$ and $u = 1/\beta$.
Hence $||u_n||$ and $||v_n||$ both converge to $1$.
It follows that
\begin{align*}
 ||Tx_n - y_n|| / ||y_n - x_n||
 \qquad\text{and}\quad 
 ||Ty_n - z_{x_n,y_n}|| / ||y_n - x_n||
\end{align*}
both converge to zero, as $n$ tends to infinity.
\end{proof}

\section{Minimal displacement, linear rate of escape, and asymptotic step size}\label{sec:4}

Let $T$ be a non-expansive mapping of a weak metric space $(X,\delta)$ into itself. We first recall the 
relations among the minimal displacement, 
\[
	\overline{\rho}(T) := \inf_{w \in X} \delta(w,Tw),
\] 
the (linear) escape rate of the orbits of $T$,
\[
	\rho(T) := \lim_{n\to\infty}\frac{\delta(x,T^{n}x)}{n},
\]
and the asymptotic step size of the orbit $(T^{n}x)_{n\geq 0}$,
\[
	\sigma_{1}(x,T) := \lim_{n\to\infty}\delta(T^{n}x,T^{n+1}x).
\]	

Non-expansiveness and the triangle inequality imply that the linear rate of escape $\rho(T) $ is well defined and 
does not depend on $x$; see the comment after \cite[Definition~11]{GV2012}. 
Moreover, the following inequality holds (see \cite[Lemma~12]{GV2012}): 
\[
	\overline{\rho}(T) \geq \rho(T). 
\]
In general, this inequality may be strict, as shown for example in~{\cite[Example~26]{GV2012}}.
Kohlberg and Neyman \cite[Theorem~1.1]{KN1981} proved that $\rho(T) = \overline{\rho}(T)$ when $(X,\delta)$ 
is a Banach space. Gaubert and Vigeral \cite[Theorem~1]{GV2012} showed that this equality also holds 
in a larger class of geodesic metric spaces, namely when $(X,\delta)$ is a complete \emph{metrically star-shaped} 
space \cite[Definition~5]{GV2012}. 

Another consequence of non-expansiveness is that, for every $x\in X$ and every integer $k\geq 1$, 
the following limit exists: 
\begin{equation}\label{eq:sigmak}
	\sigma_{k}(x,T) := \lim_{n\to\infty}\delta(T^{n}x,T^{n+k}x).
\end{equation}
Unlike the escape rate, this quantity may depend on $x$. 

By using again the triangle inequality and non-expansiveness, one can verify that 
\begin{equation}\label{eq:ineq1}
	\frac{\sigma_{k}(x,T)}{k} \leq \sigma_{1}(x,T),
	\qquad\text{for $k\ge 1$},
\end{equation}
and 
\begin{equation}\label{eq:ineq2}
	\sigma_{1}(x,T) \geq \overline{\rho}(T).
\end{equation}

Our main result is that for firm non-expansive mappings the minimal displacement, the escape rate, 
and the asymptotic step size are equal.

\begin{thm}
Let $T$ be a firm non-expansive mapping of a weak metric space $(X,\delta)$ into itself. Then, for 
every $x\in X$ and every integer $k\geq 1$,
\[
	\overline{\rho}(T) = \rho(T) = \sigma_{1}(x,T) = \frac {\sigma_{k}(x,T)} {k}.
\]
\end{thm}	

\begin{proof}
Let $x$ be a point in $X$. First, we show by induction that for every integer $k\geq 1$, 
\[ 
	\frac {\sigma_{k}(x,T)} {k} = \sigma_{1}(x,T).
\]

Let $k$ be a positive integer. Assume that the inductive hypothesis is true, that is,	
\[
	\sigma_{j}(x,T)=j\sigma_{1}(x,T),
\]
for every positive integer $j \leq k$. Let $\epsilon$ be a positive real number. It follows 
from (\ref{eq:sigmak}) and the inductive hypothesis that there exists a positive integer $N$ so that 
for every $n\geq N$ and every $j \leq k$, 
\begin{equation}\label{eq:keyineq}
	j(\sigma_{1}(x,T) - \epsilon) \leq \delta(T^{n}x,T^{n+j}x) \leq j(\sigma_{1}(x,T) + \epsilon).
\end{equation}
	
Since $T$ is firm, there exist non-negative functions $q,r,s,t$ on $X \times X$ satisfying the 
properties (\ref{eq:A}), (\ref{eq:B}) and (\ref{eq:C}). The following notation will be helpful:
\begin{align*}
	q_{a}^{b}:=q(T^{a}x,T^{b}x), & \quad & r_{a}^{b}:=r(T^{a}x,T^{b}x), & \quad & s_{a}^{b}:=s(T^{a}x,T^{b}x), \\
	t_{a}^{b}:=t(T^{a}x,T^{b}x),  & \quad & \delta_{a}^{b}:=\delta(T^{a}x,T^{b}x),
\end{align*}
for all positive integers $a,b$. Now, the property (\ref{eq:C}) implies that
\begin{equation}\label{eq:qrst1}
	\delta_{n}^{n+k+1} \geq \frac{1}{t_{n}^{n+k}}\delta_{n+1}^{n+k+1}
					- \frac{q_{n}^{n+k}}{t_{n}^{n+k}}\delta_{n}^{n+k}
					- \frac{r_{n}^{n+k}}{t_{n}^{n+k}}\delta_{n}^{n+1}
					- \frac{s_{n}^{n+k}}{t_{n}^{n+k}}\delta_{n+k}^{n+k+1}
					- \delta_{n+1}^{n+k}.
\end{equation}

By applying (\ref{eq:keyineq}) in (\ref{eq:qrst1}), it follows that for every $n\geq N$,
\begin{equation*}
\begin{split}
	\delta_{n}^{n+k+1} & \geq  \frac{k}{t_{n}^{n+k}}(\sigma_{1}(x,T) - \epsilon)
					- \frac{q_{n}^{n+k}}{t_{n}^{n+k}}k(\sigma_{1}(x,T) + \epsilon) 
				 - \frac{r_{n}^{n+k}}{t_{n}^{n+k}}(\sigma_{1}(x,T) + \epsilon)  \\
				& \qquad - \frac{s_{n}^{n+k}}{t_{n}^{n+k}}(\sigma_{1}(x,T) + \epsilon)
					- (k-1)(\sigma_{1}(x,T) + \epsilon), 
\end{split}					
\end{equation*}
or equivalently,
\begin{equation}\label{eq:bound}
	\delta_{n}^{n+k+1} 
				 \geq  \left( A_{n,k} - k + 1 \right)\sigma_{1}(x,T) 
				  - \left( B_{n,k} + k - 1 \right)\epsilon,
\end{equation}
where 
\[ A_{n,k} := \frac{k - kq_{n}^{n+k} - r_{n}^{n+k} - s_{n}^{n+k}}{t_{n}^{n+k}} \]
and
\[ B_{n,k} := \frac{k + kq_{n}^{n+k} + r_{n}^{n+k} + s_{n}^{n+k}}{t_{n}^{n+k}}. \]

On the other hand, if we let $\alpha$ denote the positive real number given in (\ref{eq:A}), then the 
property (\ref{eq:B}) implies that
\begin{equation}\label{eq:Bnkbound}
	B_{n,k} \leq \frac{1+2k}{\alpha}
\end{equation}
and also
\begin{equation}\label{eq:Ankbound}
	A_{n,k} \geq 2k.
\end{equation}
	
By applying the inequalities (\ref{eq:Bnkbound}) and (\ref{eq:Ankbound}) in (\ref{eq:bound}), we obtain for every $n\geq N$, 
\[
	\delta_{n}^{n+k+1} \geq (k+1)\sigma_{1}(x,T) - \left[\frac{1+2k}{\alpha} + k - 1\right]\epsilon.
\]
Taking the limit first as $n$ tends to infinity, and then as $\epsilon$ tends to zero, we get
\[
	\sigma_{k+1}(x,T) \geq (k+1)\sigma_{1}(x,T).
\]
This inequality together with (\ref{eq:ineq1}) show our first claim.
		
To complete the proof of our theorem, we need only show that $\sigma_{1}(x,T)$ and $\rho(T)$ are
equal. Indeed, since $T$ is non-expansive, we have that, for every positive integer $m$,
\[
	\sigma_{m}(x,T) \leq \delta(x,T^{m}x).
\]
So, by our first claim, $\sigma_{1}(x,T) \leq \delta(x,T^{m}x)/m$. By letting $m$ tend to infinity, we obtain
\[
	\sigma_{1}(x,T) \leq \rho(T).
\]
This inequality is finally combined with (\ref{eq:ineq2}) and the inequality $\overline{\rho}(T) \geq \rho(T)$ 
to complete the proof.  
\end{proof}

\section{Metric functionals}\label{sec:5}

Let $(X,\delta)$ be a metric space and let $x_0$ be an arbitrary base-point in $X$. We consider the mapping
\[
	\begin{split}
		\Phi : X &\to \mathbb{R}^{X} \\
		w &\mapsto h_w(\cdot) := \delta(\cdot,w) - \delta(x_{0},w).
	\end{split}
\]
If we endow the target space $\mathbb{R}^{X}$ with the topology of point-wise convergence, then the mapping 
$\Phi$ is a continuous injection. Moreover, the closure $\overline{\Phi(X)}$ is compact and consists of mappings 
$h : X \to \mathbb{R}$ vanishing at $x_0$ and satisfying $| h(x)-h(y) | \leq \delta(x,y)$, for every $x,y\in X$. 
Each element of the compact space $\overline{\Phi(X)}$ is called a \emph{metric functional}. Metric functionals 
and the related notion of horofunctions are discussed in \cite{G2019,GK2020,K2021,K2019}. These objects have 
been studied in spaces such as Hilbert and Thompson geometries \cite{W2008,LW2011,W2018}, Teichm\"{u}ller 
geometry \cite{W2019}, and normed spaces \cite{W2007,JS2017,G2019-0,G2019,G2020}.

Among other applications, metric functionals are useful for understanding the behaviour of the iterates of 
non-expansive mappings. Karlsson~\cite{K2001} proved the following result, which he calls the 
Metric Spectral Principle~\cite{K2021-pnas}.

\begin{thm}\label{thm:karlsson}
Let $T$ be a non-expansive mapping of a metric space $(X,\delta)$ into itself. 
Then, there exists a metric functional $h$ such that 
\[ h(T^{n}(x_{0})) \leq - \rho(T)n \] 
for every $n\geq 1$. Moreover, for every $x\in X$, 
\[ \lim_{n\to\infty} -\frac{1}{n}h(T^{n}x) = \rho(T). \]
\end{thm}

Complementing this, Gaubert and Vigeral~\cite{GV2012} proved the following.

\begin{thm}\label{thm:gaubert_vigeral}
If $(X,\delta)$ is a Banach space or, more generally, a complete metrically 
star-shaped space, then $\rho(T)=\overline{\rho}(T)$ and there exists a metric 
functional $h$ such that for every $x\in X$,
\[ h(Tx) \leq h(x)-\rho(T). \] 
\end{thm}

It should be noted that the metric functional appearing in Theorem~\ref{thm:karlsson} is a limit point of the 
orbit $(T^{n}x_{0})_{n\geq 0}$ in the compact space $\overline{\Phi(X)}$. On the other hand, in 
Theorem~\ref{thm:gaubert_vigeral}, the metric functional is constructed by composing $T$ with a retraction 
along geodesics to get a strict contraction, having of course a fixed point, and then taking a limit point of the 
sequence of fixed points as the retraction approaches the identity.

\begin{proof}[Proof of Corollary~\ref{cor:2}]
Recall that $\overline{\Phi(X)}$ is a compact space. This space is sequentially compact 
whenever $X$ is separable. In this case, there is a subsequence $n_{i}$ such that $\Phi(T^{n_{i}}x_{0})$ 
converges to some $h \in \overline{\Phi(X)}$ as $i\to\infty$. Since $T$ is a firm non-expansive 
mapping and $\delta$ is a metric, it follows from Theorem~\ref{th:1} that for every $x\in X$,
\begin{equation*}
\begin{split}
	h(Tx) & =  \lim_{i\rightarrow\infty}\big[\delta(Tx,T^{n_{i}}x_{0}) - \delta(x_{0},T^{n_{i}}x_{0}) \big]\\
		& \leq \liminf_{i\rightarrow\infty}\big[\delta(Tx,T^{n_{i}+1}x_{0}) 
				+ \delta(T^{n_{i}+1}x_{0},T^{n_{i}}x_{0}) 
				- \delta(x_{0},T^{n_{i}}x_{0}) \big]\\
		& \leq \liminf_{i\rightarrow\infty}\big[\delta(x,T^{n_{i}}x_{0})  
				- \delta(x_{0},T^{n_{i}}x_{0}) 
				+ \delta(T^{n_{i}}x_{0},T^{n_{i}+1}x_{0}) \big] \\
		& =  h(x) + \rho(T).
\end{split}
\end{equation*}
Our claim follows immediately from the assumption $\rho(T)=0$. 

In general, $\overline{\Phi(X)}$ is not sequentially compact. However, we can argue as follows. 
For every $n\geq 1$ define the set
\[
	W_{n} = \{ h \in \overline{\Phi(X)} : \text{$h(Tx)\leq h(x)+\delta(T^{n}x_{0},T^{n+1}x_{0})$, for all $x\in X$} \}.
\]
Each $W_{n}$ is non-empty since $\Phi(T^{n}x_{0})$ belongs to $W_{n}$. Moreover, $(W_{n})_{n\geq 1}$ is a 
non-increasing sequence of compact subsets of $\overline{\Phi(X)}$. Therefore, there exists 
$h\in\bigcap_{n\geq 1}W_{n}$ satisfying our claim. 	
\end{proof}	

\begin{proof}[Proof of Corollary~\ref{cor:3}]
We fix here $x_{0}=0\in V$. If the orbit $(T^{n}0)_{n\geq 0}$ is unbounded, it follows from \cite[p.~2390]{L2009} 
that there exists a subsequence $n_{i}$ such that $\Phi(T^{n_{i}}0)$ converges to some $h\in \overline{\Phi(V)}$ 
as $i\to\infty$. Moreover, for every $x\in V$,
\[ \lim_{n\to\infty}h(T^{n}x) = -\infty. \]

Since $T$ is firm non-expansive and $\rho(T)=0$, it follows from Corollary~\ref{cor:2} that 
\[ h(T^{n+1}x) \leq h(T^{n}x), \] 
for every $n\geq 0$. This completes the proof.
\end{proof}
	
\subsection*{Acknowledgment}
The first author acknowledges financial support from the Vilho, Yrj\"{o} and
Kalle V\"{a}is\"{a}l\"{a} Foundation of the Finnish Academy of Science and Letters, 
and from the Otto A. Malm Foundation.
	


\end{document}